\documentclass{amsart}

\usepackage{amscd,amsmath}
\usepackage{epic}
\usepackage[all,2cell,dvips]{xy} \UseAllTwocells \SilentMatrices
\usepackage{stmaryrd}
\usepackage{amssymb}
\usepackage{setspace} 
\usepackage{mathdots}
\usepackage{enumerate}

\newtheorem{theorem}{Theorem}[section]

\newtheorem{cor}[theorem]{Corollary}
\newtheorem{lem}[theorem]{Lemma}

\theoremstyle{definition}
\newtheorem{beisp}[theorem]{Example}
\newtheorem{definit}[theorem]{Definition}

\newtheorem*{rem*}{Remark}

\DeclareMathOperator{\Inn}{Inn}

\DeclareMathOperator{\Tr}{trace}

\DeclareMathOperator{\aut}{Aut}

\DeclareMathOperator{\id}{id}

\DeclareMathOperator{\Gl}{GL}
\DeclareMathOperator{\Sl}{SL}
\DeclareMathOperator{\Sp}{SP}

\DeclareMathOperator{\M}{M}

\renewcommand{\epsilon}{\varepsilon}
\renewcommand{\phi}{\varphi}
\renewcommand{\rho}{\varrho}
\renewcommand{\theta}{\vartheta}

\renewcommand{\le}{\leqslant}

\renewcommand{\ge}{\geqslant}

\begin{document}

\markboth{Benim, Hunnell, and Sutherland}
{Isomorphy Classes of Finite Order Automorphisms of $\Sl(2, k)$}

%
%

\title{ISOMORPHY CLASSES OF FINITE ORDER AUTOMORPHISMS OF $\Sl(2, k)$}

\author{\footnotesize ROBERT W. BENIM}

\address{Department of Mathematics and Computer Science, Pacific University\\ Forest Grove, OR
97166, USA\\
rbenim@gmail.com}

\author{MARK HUNNELL}

\address{Department of Mathematics, North Carolina State University\\
Raleigh, NC 27695, USA\\
mchunne@ncsu.edu}

\author{AMANDA K. SUTHERLAND}

\address{Department of Mathematics, North Carolina State University\\
Raleigh, NC 27695, USA\\
aksuther@ncsu.edu}

\maketitle


\begin{abstract}
In this paper, we consider the order $m$ $k$-automorphisms of $\Sl(2,k)$. We first characterize the forms that order $m$ $k$-automorphisms of $\Sl(2,k)$ take and then we simple conditions on matrices $A$ and $B$, involving eigenvalues and the field that the entries of $A$ and $B$ lie in, that are equivalent to isomorphy between the order $m$ $k$-automorphisms $Inn_A$ and $Inn_B$. We examine the number of isomorphy classes and conclude with examples for selected fields. 
\end{abstract}



\section{Introduction}

Let $G$ be a connected reductive algebraic group defined over a
field $k$ of characteristic not two, $\theta$ an involution of $G$
defined over $k$, $H$ a $k$-open subgroup of the fixed point group
of $\theta$ and $G_k$ (resp. $H_k$) the set of $k$-rational points
of $G$ (resp. $H$). The variety $G_k/H_k$ is called a symmetric
$k$-variety. For $k=\mathbb{R}$ these symmetric
$k$-varieties are also called real reductive symmetric spaces. 
These varieties occur in many problems in
representation theory, geometry and singularity theory. 
To study
these symmetric $k$-varieties one needs first a classification of
the related $k$-involutions. A characterization of the isomorphism
classes of $k$-involutions was given in \cite{Hel00}.

In \cite{HW02}, a full characterization of the isomorphism
classes of $k$-involutions was given in the case that $G=\Sl(2, k)$ which 
does not depend on any of the results in
\cite{Hel00}. 
Similarly, this is done for $\Sl(n,k)$ in \cite{HWD04}. Using this characterization, the
possible isomorphism classes for algebraically
closed fields, the real numbers, the $p$-adic numbers, and the finite
fields were classified. Analogous results for isomorphism classes of involutions of connected reductive algebraic groups can be found in \cite{Hut14} for the exceptional group $G_2$ and in \cite{BHJxx} for symplectic groups.

This concept can be generalized
 by considering order $m$ $k$-automorphisms of $G$ instead of $k$-involutions, which are of order two. We can then construct, in an analogous fashion, a generalized symmetric $k$-variety. To study
these generalized symmetric $k$-varieties, first one needs a classification of
the related order $m$ $k$-automorphisms. 

In this paper, we consider the order $m$ $k$-automorphisms of $\Sl(2,k)$ and characterize the isomorphy classes of these automorphisms. Throughout, we assume $m \ge 2$. In Section 2, we define some of the basic terminology that will be used and state previous results on the $k$-involutions of $\Sl(2,k)$. In Section 3, we characterize the form that order $m$ $k$-automorphisms of $\Sl(2,k)$ take. In Section 4, we find simple conditions on matrices $A$ and $B$, involving eigenvalues and the field that the entries of $A$ and $B$ lie in, that are equivalent to isomorphy between order $m$ $k$-automorphisms $\Inn_A$ and $\Inn_B$. In Section 5, we examine the occurrence of $m$-valid eigenpairs, which indicate an order $m$ $k$-automorphism. In Sections 6, we consider the number of isomorphy classes for a given field $k$, and order $m$. We conclude in Section 7 by examining the cases when $k = \overline{k}, \mathbb{R}$, $\mathbb{Q}$, or $\mathbb{F}_p$.

\section{Preliminaries}
We begin by defining some basic notation. Let $k$ be a field of characteristic not two, $\bar k$ the algebraic closure of $k$,
$$\M(2,k)=\{ 2\times 2 \text{-matrices with
entries in $k$} \}, $$
$$\Gl(2,k)= \{ A\in \M(2,k)\mid \det
(A)\neq 0\}$$  and
$$\Sl(2,k)= \{ A\in \M(2,k)\mid \det
(A)=1\}. $$  Let $k^*$ denote the multiplicative group of nonzero
elements of $k$, $(k^*)^2=\{a^2\mid a\in k^*\}$ denote the set of squares in $k$ and $I \in \M(2,k)$
denote the identity matrix. 

\begin{definit}
\label{isoinv} 
Let $G$ be an algebraic groups defined over a field $k$. Let $G_k$ be the $k$-rational points of $G$. Let $\aut(G, G_k)$ denote the the set of $k$-automorphisms of $G_k$. That is, $\aut(G, G_k)$ is the set of automorphisms of $G$ which leave $G_k$invariant. We say $\theta \in \aut(G,G_k)$ is a {\it $k$-involution} if $\theta^2 = \id$ but $\theta \ne \id$. A $k$-involution is a $k$-automorphism of order 2.

For $A \in G_k$, the map $\Inn_A(X) = A^{-1}XA$ is called an {\it inner $k$-automorphism of $G_k$}. We denote the set of such $k$-automorphisms by $\Inn(G_k)$. If $\Inn_A \in \Inn(G_k)$ is a $k$-involution, then we say that $\Inn_A$ is an {\it inner $k$-involution of $G_k$}.

Assume $H$ is an algebraic group defined over $k$ which contains $G$. Let $H_k$ be the $k$-rational points of $H$. For $A \in H$, if the map $\Inn_A(X) = A^{-1}XA$ is such that $\Inn_A \in \aut(G,G_k)$, then $\Inn_A$ is an {\it inner $k$-automorphism of $G_k$ over $H$}. We denote the set of such $k$-automorphisms by $\Inn(H,G_k)$. If $\Inn_A \in \Inn(H,G_k)$ is a $k$-involution, then we say that $\Inn_A$ is an {\it inner $k$-involution of $G_k$ over $H$}.

Suppose $\theta, \tau \in \aut(G,G_K)$. Then $\theta$ is {\it isomorphic} to $\tau$ {\it over $H_k$} if there is $\phi$ in ${\Inn(H_k)}$ such that $\tau=\phi ^{-1}\theta\phi$. Equivalently,  we say that $\tau$ and $\theta$ are in the same \textit{isomorphy class over $H_k$}.

\end{definit}

For simplicity, we will refer to $k$-automorphisms simply as automorphisms for the remainder of this paper.

\begin{definit}
For a field $k$, we will refer to $k^*/ (k^*)^2$ as the {\it square classes} of $k$.
\end{definit}

For example, if $k = \overline{k}$, then $|k^*/ (k^*)^2| = 1$ where 1 is a representative of this single square class. Further, $|\mathbb{R}^*/ (\mathbb{R}^*)^2| = 2$ with representatives $\pm1$; the set $\{\mathbb{Q}^*/ (\mathbb{Q}^*)^2\}$ is infinite with representatives $\pm 1$ and all the prime numbers.

The following is the main result of \cite{HW02}.

\begin{theorem}
Let $k$ be a field of characteristic not two. Then $\Sl(2,k)$ has exactly $|k^*/ (k^*)^2| $ isomorphy classes of involutions.
\end{theorem}

We will confirm this result in this paper, and see that the number of isomorphy classes of order $m$ automorphisms where $m >2$ does not depend on $|k^*/ (k^*)^2|$.


\section{Inner Automorphisms of $\Sl(2,k)$}

Since the Dynkin diagram of $\Sl(2,k)$ has a trivial automorphism group, we know that all automorphisms of $\Sl(2, k)$ are of the form $\Inn_B$ for some $B \in \Gl(2,\overline{k})$. We improve upon this fact in the following lemma.

\begin{lem}
\label{kMult}

If $\phi$ is an automorphisms of $\Sl(2,k)$, then $\phi = \Inn_A$ for some $A \in \Sl(2,k[\sqrt{\alpha}])$ where $\alpha \in k$, where each entry of $A$ is a $k$-multiple of $\sqrt{\alpha}$.

\end{lem}

\begin{proof}

Let $\phi$ be an automorphism of $\Sl(2,k)$. We can write $\phi = \Inn_B$ for some $B \in \Gl(2,\overline{k})$. It follows from Lemma 4 of \cite{HW02} that we can assume that $B \in \Gl(2,k)$. Let $A = (\det(B))^{-\frac{1}{2}}B$ and $\alpha = \det(B)$. Note that $\alpha \in k$. By construction, we see that $\det(A) = 1$ and that the entries of $A$ are $k$-multiples of $\sqrt{\alpha}$. \end{proof}

We now consider a lemma which characterizes matrices in $\Sl(2, \overline{k})$.

\begin{lem}
Suppose $A \in \Sl(2,\overline{k})$. Then $A$ is of the form 
$$A = \left(\begin{array}{cc}a & b \\-\frac{m_A(a)}{b} & -a+\lambda_1+\lambda_2\end{array}\right)$$
or 
$$A = \left(\begin{array}{cc}\lambda_1 & 0 \\c & \lambda_2\end{array}\right)$$
where $\lambda_1$ and $\lambda_2$ are the eigenvalues of $A$, and $m_A(x)$ is the minimal polynomial of $A$.
\end{lem}

\begin{proof}
If $A$ is diagonal, then $A$ is in the latter form where $c = 0$. We may assume $A$ is not diagonal and write $A =  \left(\begin{array}{cc}a & b \\c & d\end{array}\right)$. We first assume that $b$ is nonzero. We need only show that $c = -\frac{m_A(a)}{b}$ and $d = -a+\lambda_1+\lambda_2$. The latter is clear since the the trace of $A$ is $a+d = \lambda_1+\lambda_2$. So we are only concerned with $c$.

Note that $m_A(x) = x^2-\Tr(A)x+\det(A) = x^2-(\lambda_1+\lambda_2)x+1$ since $A$ is a $2 \times 2$ matrix. Now, to find the value of $c$, recall that $ad-bc =1$. Thus,

$$1 = a(-a+\lambda_1+\lambda_2)-bc,$$

which implies that 

$$bc = -a^2+(\lambda_1+\lambda_2)a-1.$$

Since $b$ is nonzero, we have that $c = -\frac{m_A(a)}{b}$. 

We now suppose $b = 0$, then $A$ is lower triangular and its diagonal entries must be its eigenvalues. Thus, $A = \left(\begin{array}{cc}\lambda_1 & 0 \\c & \lambda_2\end{array}\right)$.
\end{proof}

We can summarize the previous two lemmas into a characterization of the matrices $A \in \Sl(2,k[\sqrt{\alpha}])$ that define order $m$ automorphisms  of $\Sl(2,k)$. 

\begin{theorem}
\label{CanonThm}

Suppose $\Inn_A$ is an order $m$ automorphism of $\Sl(2,k)$ where $A \in \Sl(2,k[\sqrt{\alpha}])$, $\alpha \in k$, and each entry of $A$ is a $k$-multiple of $\sqrt{\alpha}$. Then, 

$$A = \left(\begin{array}{cc}a & b \\-\frac{m_A(a)}{b} & -a+\lambda_1+\lambda_2\end{array}\right)$$

or 

$$A = \left(\begin{array}{cc}\lambda_1 & 0 \\c & \lambda_2\end{array}\right)$$

where $\lambda_1$ and $\lambda_2$ are the eigenvalues of $A$, and $m_A(x)$ is the minimal polynomial of $A$.

\end{theorem}

\section{Isomorphy Classes of Order $m$ Automorphisms}

In this section, we find conditions on the matrices $A$ and $B$ that determine whether or not $\Inn_A$ and $\Inn_B$ are isomorphic over $\Gl(2,k)$. We begin with a lemma that translates the isomorphy conditions from one about mappings to one about matrices.

\begin{lem}
\label{TidyLem}
Assume $\Inn_A$ and $\Inn_B$ are order $m$ automorphisms of $\Sl(2,k)$. Further, suppose $A$ lies in $\Sl(2,k[\sqrt{\alpha}])$ where each entry of $A$ is a $k$-multiple of $\sqrt{\alpha}$, $B$ lies in $\Sl(2,k[\sqrt{\gamma}])$ where each entry of $B$ is a $k$-multiple of $\sqrt{\gamma}$, where $\alpha, \gamma \in k$. Then $\Inn_A$ and $\Inn_B$ are isomorphic over $\Gl(2,k)$ if and only if there exists $Q \in \Gl(2,k)$ such that $Q^{-1}AQ = B$ or $-B$.
\end{lem}

\begin{proof}
 First assume there exists $Q \in \Gl(2,k)$ such that $Q^{-1}AQ = B$ or $-B$. Then for all $U \in \Sl(2,k)$, we have 
\begin{align*}
\Inn_Q\Inn_A \Inn_{Q^{-1}} (U) &= Q^{-1}A^{-1}QUQ^{-1}AQ\\ 
&= (Q^{-1}AQ)^{-1}U(Q^{-1}AQ)\\
&= (\pm B)^{-1}U(\pm B)\\ 
&= B^{-1}UB\\ 
&= \Inn_B(U).
\end{align*}
So, $\Inn_Q\Inn_A \Inn_{Q^{-1}} =\Inn_B$ and $\Inn_A$ and $\Inn_B$ are isomorphic over $\Gl(2,k)$.

To prove the converse, we now assume that $\Inn_A$ and $\Inn_B$ are isomorphic over $\Gl(2,k)$. Then there exists $Q \in \Gl(2,k)$ such that $\Inn_Q\Inn_A \Inn_{Q^{-1}} =\Inn_B$. We note that $\Inn_A$ and $\Inn_B$ are also automorphisms of $\Sl(2,\overline{k})$. For all $U \in \Sl(2,\overline{k})$, we have 

$$Q^{-1}A^{-1}QUQ^{-1}AQ = B^{-1}UB,$$

which implies 

$$BQ^{-1}A^{-1}QUQ^{-1}AQB^{-1} = U.$$

So, $Q^{-1}AQB^{-1} $ commutes with all elements of $\Sl(2,\overline{k})$. We note that $Q^{-1}AQB^{-1}  \in \Sl(2,\overline{k})$, so $Q^{-1}AQB^{-1} $ must lie in the center of $\Sl(2,\overline{k})$, which is $\{ I, -I \}$. Thus $Q^{-1}AQ = B$ or $-B$. 
\end{proof}

Note that $\Inn_A$ and $\Inn_B$ will be isomorphic only if $A$ and $B$ have entries in the same quadratic extension of $k$.

\begin{lem}
\label{SquareClassLem}
Assume $\Inn_A$ and $\Inn_B$ are order $m$ automorphisms of $\Sl(2,k)$, $A$ lies in $\Sl(2,k[\sqrt{\alpha}])$ where each entry of $A$ is a $k$-multiple of $\sqrt{\alpha}$, and $B$ lies in $\Sl(2,k[\sqrt{\gamma}])$ where each entry of $B$ is a $k$-multiple of $\sqrt{\gamma}$, where $\alpha, \gamma \in k$. If $\Inn_A$ and $\Inn_B$ are isomorphic over $\Gl(2,k)$, then $\gamma = c\alpha$. That is, $\alpha$ and $\gamma$ lie in the same square class of $k$, and all of the entries of $B$ are $k$-multiples of $\sqrt{\alpha}$.
\end{lem}

\begin{proof}
By Lemma \ref{TidyLem}, there exists $Q \in \Gl(2,k)$ such that $Q^{-1}AQ = B$ or $-B$ and the result follows.
\end{proof}

Using the previous theorem and lemmas, we can now characterize isomorphy classes of order $m$ automorphisms of $\Sl(2,k)$.

\begin{theorem}
\label{IsomorphyThm}
Suppose $\Inn_A$ and $\Inn_B$ are order $m$ automorphisms of $\Sl(2,k)$ where $A$ and $B \in \Sl(2,k[\sqrt{\alpha}])$ for some $\alpha \in k$ where each entry of $A$ and $B$ is a $k$-multiple of $\sqrt{\alpha}$.
\begin{enumerate}[(a)]
\item If $A$ and $B$ have the same eigenvalues, $\lambda_1$ and $\lambda_2$, then, $\Inn_A$ and $\Inn_B$ are isomorphic over $\Gl(2,k)$.
\item If $A$ has eigenvalues $\lambda_1$ and $\lambda_2$ and $B$ has eigenvalues $-\lambda_1$ and $-\lambda_2$, then $\Inn_A$ and $\Inn_B$ are isomorphic over $\Gl(2,k)$.
\item If $\Inn_A$ is isomorphic to $\Inn_B$ over $\Gl(2,k)$, then $A$ has the same eigenvalues as $B$ or $-B$.
\end{enumerate}
\end{theorem}

\begin{proof}
\begin{enumerate}[(a)]
\item We consider two cases based on if $\lambda_1$ and $\lambda_2$ are $k$-multiples of $\sqrt{\alpha}$.

{\bf Case 1:}  If $\lambda_1$ and $\lambda_2$ are not $k$-multiples of $\sqrt{\alpha}$, then both $A$ and $B$ must not be lower triangular. We can assume 

$$A = \left(\begin{array}{cc}a & b \\-\frac{m_A(a)}{b} & -a+\lambda_1+\lambda_2\end{array}\right)$$ and  $$B = \left(\begin{array}{cc}c & d \\-\frac{m_A(c)}{d} & -c+\lambda_1+\lambda_2\end{array}\right).$$ 

Then for

$$Q_A = \left(\begin{array}{cc}b & b \\\lambda_1-a & \lambda_2-a\end{array}\right) \in \Gl(2,\overline{k}),$$

we have 

$$Q_A^{-1}AQ_A = \left(\begin{array}{cc}\lambda_1 & 0 \\0 & \lambda_2\end{array}\right).$$
 
 Likewise, if we let
 $$Q_B = \left(\begin{array}{cc}d & d \\\lambda_1-c & \lambda_2-c\end{array}\right) \in \Gl(2,\overline{k}),$$

it follows that

$$Q_B^{-1}BQ_B = \left(\begin{array}{cc}\lambda_1 & 0 \\0 & \lambda_2\end{array}\right).$$

Let 

$$Q = Q_AQ_B^{-1} = \left(\begin{array}{cc}\frac{b}{d} & 0 \\\frac{c-a}{d} & 1\end{array}\right).$$

Note that $Q^{-1}AQ = B$ and that $Q \in \Gl(2,k)$. Using the result of Lemma \ref{TidyLem}, we have shown that $\Inn_A$ and $\Inn_B$ are isomorphic over $\Gl(2,k)$.

{\bf Case 2:}  Let $\lambda_1$ and $\lambda_2$ be $k$-multiples of $\sqrt{\alpha}$ and define $D =  \left(\begin{array}{cc}\lambda_1 & 0 \\0 & \lambda_2\end{array}\right)$. In this case, it is possible but not necessary that $A$ and $B$ are lower triangular. If neither are triangular, then the argument from Case 1 shows that $\Inn_A$ and $\Inn_B$ are isomorphic over $\Gl(2,k)$, as desired. Assume that $A$ and $B$ are lower triangular. We write

$$A = \left(\begin{array}{cc}\lambda_1 & 0 \\c & \lambda_2\end{array}\right).$$

From Lemma \ref{kMult}, we know that $\lambda_1, \lambda_2$, and $c$ are $k$-multiples of $\sqrt{\alpha}$. Let

$$Q_A = \left(\begin{array}{cc} \frac{\lambda_1-\lambda_2}{c} & 0 \\ 1& 1\end{array}\right) \in \Gl(2,k)$$

then

$$Q_A^{-1}AQ_A = \left(\begin{array}{cc}\lambda_1 & 0 \\0 & \lambda_2\end{array}\right) = D.$$ 

Since $A$ induces an order $m$ automorphism of $\Sl(2,k)$, $D= \left(\begin{array}{cc}\lambda_1 & 0 \\0 & \lambda_2\end{array}\right)$ must induce an order $m$ automorphism of $\Sl(2,k)$, $\Inn_D$. We have shown that $\Inn_D$ is isomorphic over $\Gl(2,k)$ to $\Inn_A$ by Lemma \ref{TidyLem}.

If $B$ is lower triangular as well, then we can show that $\Inn_B$ is isomorphic to the automorphism induced by $\Inn_D$. By transitivity of isomorphy, $\Inn_A$ is isomorphic to $\Inn_B$ over $\Gl(2,k)$.

The only case left to consider is when $A$ is not lower triangular, but $B$ is lower triangular. It suffices to show that $\Inn_A$ is isomorphic over $\Gl(2,k)$ to $\Inn_D$, since we have already shown $\Inn_B$ is isomorphic to $\Inn_D$. We again consider 

$$A = \left(\begin{array}{cc}a & b \\-\frac{m_A(a)}{b} & -a+\lambda_1+\lambda_2\end{array}\right) \in \Sl(2,k[\sqrt{\alpha}])$$ and   $$Q_A = \left(\begin{array}{cc}b & b \\\lambda_1-a & \lambda_2-a\end{array}\right) \in \Gl(2,\overline{k}),$$

where

$$Q_A^{-1}AQ_A = \left(\begin{array}{cc}\lambda_1 & 0 \\0 & \lambda_2\end{array}\right) = D.$$

Let $Q_2 = \sqrt{\alpha}Q_A.$ Since all of the entries of $Q_A$ are $k$-multiples of $\sqrt{\alpha}$, it follows that $Q_2 \in \Gl(2,k)$. We can see that $Q_2^{-1}AQ_2 = \left(\begin{array}{cc}\lambda_1 & 0 \\0 & \lambda_2\end{array}\right) = D,$ and therefore $\Inn_A$ is isomorphic to $\Inn_D$ by Lemma \ref{TidyLem}. 

\item Suppose $A$ has eigenvalues $\lambda_1$ and $\lambda_2$ and $B$ has eigenvalues $-\lambda_1$ and $-\lambda_2$. Observe that $A$ and $-B$ have the same eigenvalues. From the proof of $(a)$, we know that $\Inn_A$ is isomorphic to $\Inn_{-B}$. Since $\Inn_B = \Inn_{-B}$, we are done.

\item Suppose $\Inn_A$ is isomorphic to $\Inn_B$ over $\Gl(2,k)$. By Lemma \ref{TidyLem}, there exists $Q \in \Gl(2,k)$ such that $Q^{-1}AQ = B$ or $-B$.

\end{enumerate}
\end{proof}

We summarize the results of this theorem in the following corollary.

\begin{cor}
\label{IsomorphyCor}
Suppose $\Inn_A$ and $\Inn_B$ are order $m$ automorphisms of $\Sl(2,k)$ where $A$ and $B \in \Sl(2,k[\sqrt{\alpha}])$ for some $\alpha \in k$ and each entry of $A$ and $B$ is a $k$-multiple of $\sqrt{\alpha}$. Then $\Inn_A$ is isomorphic to $\Inn_B$ over $\Gl(2,k)$ if and only if $A$ has the same eigenvalues as $B$ or $-B$.
\end{cor}


\section{$m$-Valid Eigenpairs}

In the previous section, we reduced the problem of isomorphy to a problem of eigenvalues and quadratic extensions. In this section, we consider the valid pairs of eigenvalues of a matrix $A$ that could induce an automorphism of order $m$.

\begin{definit}
We call the pair $\lambda_1$, $\lambda_2 \in \overline{k}$ an {\it $m$-valid eigenpair} if $\Inn_A$ is an order $m$ automorphism of $\Sl(2, \overline{k})$ where $A = \left[\begin{array}{cc}\lambda_1 & 0 \\0 & \lambda_2\end{array}\right] \in \Sl(2,\overline{k})$.
\end{definit}

In the following two lemmas we characterize the matrices $B$ where $\Inn_B$ acts as the identity on $\Sl(2,k)$. 

\begin{lem}\label{helmwu}
Suppose $\Inn_B$ for $B \in \Gl(n,\overline{k})$ acts as the identity on $\Sl(2,k)$. Then $B = c I$ for some $c\in \overline{k}$.
\end{lem}

\begin{proof}

This is Lemma 2 of \cite{HW02}.

\end{proof}

We can improve upon this statement since we can assume $B \in \Sl(2,\overline{k})$. We can use this idea to characterize the matrices that induce order $m$ automorphisms on $\Sl(2,k)$.

\begin{lem}
\label{EigenpairLem}
\begin{enumerate}[(a)]
\item Suppose $\Inn_B$ for $B \in \Sl(2,\overline{k})$ acts as the identity on $\Sl(2,k)$. Then $B = I$ or $B = -I$. 
\item $\Inn_A$ is an order $m$ automorphism of $\Sl(2,k)$ if and only if $m$ is the smallest integer such that $A^m = I$ or $A^m = -I$.
\end{enumerate}
\end{lem}

\begin{proof}
\begin{enumerate}[(a)]
\item From Lemma \ref{helmwu}, we have that $B = cI$ for some $c\in \overline{k}$. Since $B \in \Sl(2,\overline{k})$, $\det(B) =1= c^2$, which means $c = \pm 1.$ 
\item If $m$ is the smallest integer such that $A^m = I$ or $A^m = -I$, then $m$ is the smallest integer such that $\Inn_{A^m} = (\Inn_A)^m$ acts as the identity on $\Sl(2,k)$, which means $\Inn_A$ is an order $m$ automorphism of $\Sl(2,k)$.

If $\Inn_A$ is an order $m$ automorphism of $\Sl(2,k)$, then $\Inn_{A^m}$ acts as the identity on $\Sl(2,k)$. $(a)$ implies that $A^m = I$ or $A^m = -I$. If there exists $r$ such that $0 \le r < m$ where $A^r = I$ or $A^r = -I$, then $\Inn_A$ is at most an order $r$ automorphism of $\Sl(2,k)$, which is a contradiction. Thus, $m$ is the smallest integer such that $A^m = I$ or $A^m = -I$.
\end{enumerate}
\end{proof}

We can characterize the $m$-valid eigenpairs.

\begin{theorem}
\label{EigenpairThm}
$\lambda_1$ and $\lambda_2$ are an $m$-valid eigenpair if and only if 
\begin{enumerate}[(a)]
\item $\lambda_1$ is a primitive $2m$-th root of unity and $\lambda_2 = \lambda_1^{2m-1}$, or 
\item $m$ is odd, $\lambda_1$ is a primitive $m$-th root of unity and $\lambda_2 = \lambda_1^{m-1}$
\end{enumerate}

\end{theorem}

\begin{proof}

Let $A = \left[\begin{array}{cc}\lambda_1 & 0 \\0 & \lambda_2\end{array}\right]$. We begin by proving necessity, so assume that $\Inn_A$ is an order $m$ automorphism of $\Sl(2,\overline{k})$. We may assume that $A \in \Sl(2,\overline{k})$ by Lemma \ref{kMult}. By Lemma \ref{EigenpairLem} $(b)$, we know that $m$ is the smallest integer such that $A^m = I$ or $A^m = -I$. There are two cases to consider.

First assume that $m$ is the smallest integer that $A^m = -I$ and that $A^r \ne I$ when $0 \le r \le m$. Then $\lambda_1$ is a $2m$-th root of unity. Since $\det(A)= 1$, $\lambda_2 = \lambda_1^{2m-1}$.

Now assume that $m$ is the smallest integer such that $A^m = I$ and that $A^r \ne -I$ when $0 \le r \le m$. Then $\lambda_1$ is an $m$-th root of unity. Since $\det(A)= 1$, then $\lambda_2 = \lambda_1^{m-1}$. 

Now we prove the sufficiency of the conditions. In either case, $A \in \Sl(2, \overline{k})$ follows from the construction of $A$. Let's first assume $(a)$, then $m$ is the smallest positive integer such that $\lambda_1^m = -1 = \lambda_2^m$, and $2m$ is the smallest integer such that $\lambda_1^{2m} = -1 = \lambda_2^{2m}$. Thus, $m$ is the smallest integer such that $A^m = -I$ and $2m$ is the smallest integer such that $A^{2m} = I$. By Lemma \ref{EigenpairLem} $(b)$, $\Inn_A$ is an order $m$ automorphism of $\Sl(2, \overline{k})$.

Now assume the conditions of $(b)$. Then $m$ is the smallest integer such that $\lambda_1^m = 1 = \lambda_2^m$, and for every integer $r$ where $0 \le r < m$. We know that $\lambda_1^r \ne -1$, so $m$ is the smallest integer such that $A^m = I$, and Lemma \ref{EigenpairLem} $(b)$ tells us that $\Inn_A$ is an order $m$ automorphism of $\Sl(2, \overline{k})$.
\end{proof}

Let $\phi$ denote Euler's $\phi$-function. That is, for positive integer $m$, $\phi(m)$ is the number of integers $l$ such that $1 \le l < m$ and $\gcd(l,m) = 1$. 

\begin{cor}
\label{EigenpairCor}
For any given field $k$, there are $\phi(m)$ $m$-valid eigenpairs.
\end{cor}

\begin{proof}

We consider separately the cases where $m$ is odd and even.
First, assume $m$ is even. Write $m = 2^s t$ where $s$ and $t$ are integers and $t$ is odd. If we include ordering, then there are $\phi(2m)$ such pairs. This double counts the $m$-valid eigenpairs. Thus, the number of distinct $m$-valid eigenpairs is 

\begin{align*}
\frac{\phi(2m)}{2} &= \frac{\phi(2^{s+1}t)}{2}\\
&= \frac{\phi(2^{s+1})\phi(t)}{2}\\
&= \frac{2^{s}\phi(t)}{2}\\
&= 2^{s-1}\phi(t)\\
&= \phi(2^s)\phi(t)\\
&= \phi(2^st)\\
&=\phi(m).
\end{align*}

Now suppose $m$ is odd. The eigenvalues may be primitive $m$-th or $2m$-th roots of unity. If we include ordering, there are $\phi(m)+\phi(2m)$ such pairs. Again, this double counts the $m$-valid eigenpairs. The number of distinct $m$-valid eigenpairs when $m$ is odd is 

\begin{align*}
\frac{\phi(m)+\phi(2m)}{2} &= \frac{\phi(m)+\phi(m)}{2}\\
&=\phi(m).
\end{align*}

Regardless of the parity of $m$, there are always $\phi(m)$ $m$-valid eigenpairs.

\end{proof}

\section{Number of Isomorphy Classes}

Given a field $k$, not necessarily algebraically closed, we would like to know the number of the isomorphy classes of order $m$ automorphisms of $\Sl(2,k)$. 

\begin{definit}
Let $C(m,k)$ denote the number of isomorphy classes of order $m$ automorphisms of $\Sl(2,k)$.
\end{definit}

\begin{theorem}
\label{NumberClassesThm}
$C(m,k) = \frac{1}{2}\phi(m)$ or $0$ for $m >2$, and $C(2,k) = |k^*/(k^*)^2|$.
\end{theorem}

\begin{proof}
From Corollary 2 in \cite{HW02}, we know that $C(2,k) = |k^*/(k^*)^2|$. This is also clear from our results, since there is exactly one $2$-valid eigenpair, consisting of the two roots of $-1$.

Now assume $m > 2$. We claim that each $m$-valid eigenpair induces either one or zero isomorphy classes. Recall that if $\Inn_A$ is an order $m$ automorphism, then by Theorem \ref{CanonThm} we may assume that $\lambda$ is an $m$-th or $2m$-th primitive root of unity and 
$$A = \left(\begin{array}{cc}a & b \\-\frac{m_A(a)}{b} & -a+\lambda+\lambda^{-1}\end{array}\right)$$

or 

$$A = \left(\begin{array}{cc}\lambda & 0 \\c & \lambda^{-1} \end{array}\right),$$

where $\det(A) = 1$ and the entries of $A$ are in $k$, or are $k$-multiples of $\sqrt{\alpha}$ for some $\alpha \in k$. If $\lambda+\lambda^{-1}$ is nonzero, then $\lambda+\lambda^{-1}$ can lie in at most one square class of $k$. We need only show that $\lambda+\lambda^{-1} \ne 0$ when $m >2$. If $\lambda+\lambda^{-1} = 0$, then we can rearrange this equation to get $\lambda^2 = -1$, which is the case only when $m = 2$. 

In Corollary \ref{EigenpairCor}, we showed that there are always $\phi(m)$ $m$-valid eigenpairs. It follows from Corollary \ref{IsomorphyCor} that if $\Inn_A$ and $\Inn_B$ are isomorphic where $A, B \in \Sl(2,k[\sqrt{\alpha}])$, then $A$ has the same eigenvalues as $B$ or $-B$. So, $\Inn_A$ and $\Inn_{-A}$ are isomorphic. If $A$ has eigenvalues $\lambda$ and $\lambda^{-1}$, then $-A$ has eigenvalues $-\lambda$ and $-\lambda^{-1}$. Therefore, exactly two $m$-valid eigenpairs induce the same isomorphy class of order $m$ automorphisms of $\Sl(2,k)$, assuming the isomorphy classes exist.
\end{proof}

For the remainder of this section, we consider how many quadratic extensions of $k$ can induce an order $m$ automorphism of $\Sl(2,k)$, specifically when $m >2$.

\begin{lem}
\label{EigenPowerLem}
Let $k$ be a field, $\alpha\in k$, and suppose $\lambda$ is an $l$th primitive root of unity.
\begin{enumerate}[(a)]
\item If $\lambda$ is a $k$-multiple of $\sqrt{\alpha}$, then so is $\lambda^r$ for all odd integers $r$, and $\lambda^r \in k$ for all even integers $r$.
\item If $\lambda + \lambda^{-1}$ is a $k$-multiple of $\sqrt{\alpha}$, then so is $\lambda^r+\lambda^{-r}$ for all odd integers $r$ and $\lambda^r +\lambda^{-r} \in k$ for all even integers $r$.
\end{enumerate}
\end{lem}

\begin{proof}

The proof of $(a)$ is clear. We probe $(b)$ by induction.
Let $r>1$ be even and suppose $\lambda + \lambda^{-1}$ and $\lambda^{r-1} + \lambda^{-(r-1)}$ are $k$-multiples of $\sqrt{\alpha}$, and that $\lambda^{r-2} + \lambda^{-(r-2)} \in k$. Then 
$$(\lambda + \lambda^{-1})(\lambda^{r-1} + \lambda^{-(r-1)}) = (\lambda^r+\lambda^{-r})+(\lambda^{r-2} + \lambda^{-(r-2)}) \in k.$$ Thus, $\lambda^r +\lambda^{-r} \in k$.

Let $r>1$ be odd and suppose $\lambda + \lambda^{-1}$ and $\lambda^{r-2} + \lambda^{-(r-2)}$ are $k$-multiples of $\sqrt{\alpha}$, and that $\lambda^{r-1} + \lambda^{-(r-1)} \in k.$ Then an argument similar to the above shows that $\lambda^r+\lambda^{-r}$ is a $k$-multiple of $\sqrt{\alpha}$.
\end{proof}

From Theorem \ref{NumberClassesThm}, if $m >2$, then each $m$-valid eigenpair can induce at most one isomorphy class of order $m$ automorphisms of $\Sl(2,k)$. Paired Lemma \ref{EigenPowerLem}, if $\Sl(2,k)$ has an order $m$ automorphism $\Inn_A$, then the entries of matrices $A$ that induce these automorphisms will have entries in $k$, or a single quadratic extension of $k$. This gives the following result.

\begin{cor}
If $m >2$ and $det(A) = 1 = \det(B)$, then it is not possible for $\Inn_A$ and $\Inn_B$ to be order $m$ automorphisms and for $A$ and $B$ to have entries in distinct quadratic extensions of $k$.
\end{cor}
\section{Examples}

We now look at a few examples over different fields $k$.

\begin{beisp}[$k = \overline{k}$]

Since all roots of unity will lie in $k$ when $k$ is algebraically closed, then every $m$-valid eigenpair, $(\lambda_1, \lambda_2)$, will induce an order $m$ automorphism of $\Sl(2,k)$ of the form $\Inn_A$ where $A = \left(\begin{array}{cc}\lambda_1 & 0 \\0 & \lambda_2\end{array}\right)$. The following results from Theorem \ref{NumberClassesThm}:

\begin{theorem}
  $C(2,\overline{k}) =1 $ and $C(m,\overline{k}) = \frac{1}{2}\phi(m)$ when $m >2$.
 \end{theorem}
\end{beisp}
\begin{beisp}[$k = \mathbb{R}$]

Let $i$ denote the square root of -1 and $\lambda$ be an $l$th primitive root of unity, where we assume $l = 2m$ or $l = m$ and $m$ is odd. We know that $(\lambda, \lambda^{l-1})$ is an $l$-valid eigenpair by Theorem \ref{EigenpairThm}. For this eigenpair to induce an automorphism on $\Sl(2,\mathbb{R})$, we need one of the following to be the case:
\begin{enumerate}[(a)]
\item $\lambda \in \mathbb{R}$;
\item $\lambda = \gamma i$, for $\gamma \in \mathbb{R}$;
\item $\lambda + \lambda^{l-1} \in \mathbb{R}$; or
\item $\lambda + \lambda^{l-1} = \gamma i$, for $\gamma \in \mathbb{R}$.
\end{enumerate}

These conditions follow since the entries of $A$ must lie in $\mathbb{R}$ or be $\mathbb{R}$-multiples of $i$. $(a)$ and $(b)$ correspond to $A = \left(\begin{array}{cc}\lambda & 0 \\ c & \lambda^{l-1}\end{array}\right)$ inducing the automorphism $\Inn_A$, and $(c)$ and $(d)$ correspond to $A = \left(\begin{array}{cc}a & b \\-\frac{m_A(a)}{b} & -a+\lambda+\lambda^{l-1}\end{array}\right)$ also inducing  the automorphism $\Inn_A$. Further, $(a)$ and $(c)$ correspond to the entries of $A$ falling in $\mathbb{R}$, and $(b)$ and $(d)$ correspond to the entries of $A$ being $\mathbb{R}$-multiples of $i$. Using De Moivre's formula, we can write $$\lambda = \cos \left( \frac{2\pi r}{l} \right) + i \sin \left( \frac{2\pi r}{l} \right)$$ and  $$\lambda^{l-1} = \cos \left( \frac{2\pi r}{l} \right) - i \sin \left( \frac{2\pi r}{l} \right)$$ for some integer $r$ where $0 < r < l$ and $r$ is coprime to $l$.  We can easily check to see when we have each of the four cases listed above.

\begin{enumerate}[(a)]
\item When is $\lambda \in \mathbb{R}$? If $\lambda \in \mathbb{R}$, then $\lambda^h \in \mathbb{R}$ for all integers $h$. So we may assume that $r = 1$. This will occur when $\sin \left( \frac{2\pi}{l} \right)= 0$. Thus, $l = 2$ and $\lambda = -1$. Since we are assuming $m \ge 2$, this cannot happen.
\item When is $\lambda = \gamma i$, for $\gamma \in \mathbb{R}$? Similar to the previous case, we may assume that $r=1$. Then $\lambda = \gamma i$, for $\gamma \in \mathbb{R}$ will occur when $\cos \left( \frac{2\pi}{l} \right)= 0$. This can happen only when $\frac{2\pi}{l} = \frac{\pi}{2}$ or $\frac{3\pi}{2}$, which yields $l = 4$ and $l = \frac{4}{3}$, respectively. The latter solution does not concern us, but the solution $l = 4$ occurs if $\lambda = i$. This happens when $m = 2$, and there is one $2$-valid eigenpair, $(i,-i)$.
\item When is $\lambda + \lambda^{l-1} \in \mathbb{R}$? Using De Moivre's formula, we see that 
\begin{align*}
\lambda+\lambda^{l-1} &= \left( \cos \left( \frac{2\pi r}{l} \right) + i \sin \left( \frac{2\pi r}{l} \right) \right)+ \left(\cos \left( \frac{2\pi r}{l} \right) - i \sin \left( \frac{2\pi r}{l} \right) \right) \\
&= 2\cos \left( \frac{2\pi r}{l} \right) \in \mathbb{R}.
\end{align*}
This is always the case.
\item Based on the previous case, we see that $\lambda + \lambda^{l-1} = \gamma i$ for $\gamma \in \mathbb{R}$ is never the case.
\end{enumerate}

If $m = 2$, then $l = 4$. There are two isomorphy classes of order 2 automorphisms: one where the matrix takes entries in $\mathbb{R}$ from $(c)$, and one where the matrix has entries that are $\mathbb{R}$-multiples of $i$ from case $(b)$. Thus, $C(2,\mathbb{R})= 2$, which agrees with the results in \cite{HW02} and Theorem \ref{NumberClassesThm}. 

Suppose $m >2$. Case $(c)$ applies here. It follows that there are always $m$th and $2m$th primitive roots of unity. We have the following result.

\begin{theorem}
If $m = 2$, then $C(2, \mathbb{R}) = 2$; if $m > 2$, then $C(m,\mathbb{R}) = \frac{1}{2}\phi(m).$
\end{theorem}
\end{beisp}

\begin{beisp}[$k = \mathbb{Q}$]

We know that $C(2,\mathbb{Q})$ is infinite. Consider the case where $m>2$. As noted in the case where $k = \mathbb{R}$, if $\lambda$ is an $l$th root of unity where $l = m$ or $2m$, then $\lambda+\lambda^{-1} = 2\cos \left( \frac{2\pi r}{l} \right)$. $\Sl(2,\mathbb{Q})$ will have order $m$ automorphisms if and only if $\cos \left( \frac{2\pi r}{l} \right)$ lies in $\mathbb{Q}$ or is a $\mathbb{Q}$ multiple of $\sqrt{p}$ for some prime $p$.

We first examine the case when $\cos \left( \frac{2\pi r}{l} \right)$ lies in $\mathbb{Q}$. By Niven's Theorem, Corollary 3.12 of \cite{Niv56}, $\cos x$ and $\frac{x}{\pi}$ are simultaneously rational only when $\cos x = 0, \pm \frac{1}{2}$, or $\pm 1$. By Lemma \ref{EigenPowerLem}, we may assume $r = 1$. Then $\cos \left( \frac{2\pi}{l} \right)$ is rational if and only if $l = 6, 4, 3, 2, \frac{3}{2}, \frac{4}{3}$, or $\frac{6}{5}$. Since $l$ must be an integer, we need only consider $l = 6, 4, 3,$ or 2. Since $m >2 $ we can further restrict our considerations to $l = 3$ or 6. Both of these correspond to order 3 automorphisms. There is $\frac{\phi(3)}{2} = 1$ isomorphy class of order 3 automorphisms of $\Sl(2,\mathbb{Q})$. If we let $l= 6$ and choose $a = b= 1$, then $$A = \left(\begin{array}{cc}a & b \\-\frac{m_A(a)}{b} & -a+\lambda+\lambda^{l-1}\end{array}\right) = \left(\begin{array}{cc}1 & 1 \\-1 & 0\end{array}\right)$$ is a matrix that will induce an order 3 automorphism.

We now consider the case when $2\cos \left( \frac{2\pi r}{l} \right)$ is a $\mathbb{Q}$ multiple of $\sqrt{p}$ for some prime number $p$. Again, it is sufficient to consider the case where $r = 1$. We note the following lemma which is a part of Theorem 3.9 in \cite{Niv56}.

\begin{lem}
Let $l$ be a positive integer. Then $2\cos \left( \frac{2\pi}{l} \right)$ is an algebraic integer which satisfies a minimal polynomial of degree $\frac{\phi(l)}{2}$.
\end{lem}

Since we are interested in knowing when $2\cos \left( \frac{2\pi r}{l} \right) = \mu \sqrt{p}$ for some $\mu \in \mathbb{Q}$ and prime $p$, we need $2\cos \left( \frac{2\pi r}{l} \right)$ to satisfy a polynomial of the form $x^2-\mu^2p = 0$. A necessary condition for such $l$ is that $\frac{\phi(l)}{2} = 2$, or $\phi(l) = 4.$

If $l = p^m$ for some prime $p$, then $$4 = \phi(p^m) = p^{m-1}(p-1).$$ Note that $p$ and $p-1$ cannot both be even, so it must be the case that $p^{m-1} = 4$ and $p-1 = 1$, which means $l = 8$, or $p^{m-1} = 1$ and $p-1 = 4$, which means $l = 5$. If $l = p^mq^t$ for some distinct primes $p$ and $q$, then $$4 = \phi(p^mq^t) = (p^m-p^{m-1})(q^t-q^{t-1}).$$ If $p^m-p^{m-1} = 2 = q^t-q^{t-1}$, then $p^m = 4$ and $q^t = 3$ which means $l = 12$. (Other primes and/or larger powers would not yield $\phi(p^m) = 2.$) If $p^m-p^{m-1} = 4$ and $q^t-q^{t-1} = 1$, then $p^m = 8$ or 5, and $q^t = 2$. Since $p$ and $q$ are distinct, we have $l = 10$. If $l$ is a multiple of three or more distinct primes, then $\phi(l) >4$. So, the only $l$ for which $\phi(l) = 4$ are $l= 5, 8, 10$ and 12. Note that 

$$2\cos \left( \frac{2\pi }{5} \right) = \frac{-1+\sqrt{5}}{2},$$

$$2\cos \left( \frac{2\pi }{8} \right) = \sqrt{2},$$

$$2\cos \left( \frac{2\pi }{10} \right) = \frac{1+\sqrt{5}}{2},$$

and 

$$2\cos \left( \frac{2\pi }{12} \right) = \sqrt{3}.$$

When $l = 8$ or 12, $2\cos \left( \frac{2\pi r}{l} \right)$ satisfies a polynomial of the form $x^2-\mu^2p = 0$, but no linear polynomial and for no other values of $l$. Thus, $\Sl(2, \mathbb{Q})$ also has automorphisms of order 4 and 6.

\begin{theorem}
$\Sl(2, \mathbb{Q})$ only has finite order automorphisms of orders 1, 2, 3, 4, and 6. Further, $C(2, \mathbb{Q})$ is infinite, and $C(3, \mathbb{Q}) = C(4, \mathbb{Q}) = C(6, \mathbb{Q}) = 1$.
\end{theorem}
\end{beisp}
\begin{beisp}[$k = \mathbb{F}_q$, $q = p^r$, $p \ne 2$]
If $m = 2$, then $C(2,\mathbb{F}_q) = 2$. Again, assume $m >2$. We need only determine when $m$th and $2m$th primitive roots of unity lie in $\mathbb{F}_q$ or are an $\mathbb{F}_q$-muliple of $\sqrt{\alpha}$ for some $\alpha \in \mathbb F_q$. We first consider the primitive roots which lie in $\mathbb{F}_q$. It is known that $\mathbb{F}_q \setminus \{ 0 \}$ is a cyclic multiplicative group of order $q-1$, so it contains elements of orders $q-1$, and all of $(q-1)$'s divisors. Thus, $\mathbb{F}_q $ will contain all of the primitive roots of unity of orders $q-1$ and its divisors.

We now consider the primitive roots of unity which are $\mathbb{F}_q$ multiples of $\sqrt{\alpha}$ for some $\alpha \in \mathbb{F}_q$. Suppose $\lambda = \mu \sqrt{\alpha}$ where $\mu$, $\alpha \in \mathbb{F}_q$. Note that $$\lambda^{q-1} = \mu^{q-1}\alpha^{\frac{q-1}{2}} = \alpha^{\frac{q-1}{2}}.$$ It follows that $\lambda^{2(q-1)} = 1$. The maximal possible value $l$ such that an $l$th primitive root of unity is an $\mathbb{F}_q$ multiple of $\sqrt{\alpha}$ for $\alpha \in \mathbb{F}_q$ is $2(q-1)$. To see that this maximal order of primitive roots of unity will always occur, suppose $\alpha \in \mathbb{F}_q$ is a $(q-1)$th  primitive root of unity. Then $\sqrt{\alpha}$ is a $2(q-1)$th primitive root of unity. This, along with Theorem \ref{NumberClassesThm} proves the following result.

\begin{theorem}
\begin{enumerate}[(a)]

\item If $m = 2$, then $C(2,\mathbb{F}_q) = 2$. 

\item If $m>2 $ is even and $2m$ divides $2(q-1)$, or if $m$ is odd and $m$ (and $2m$) divides $q-1$, then $C(m,\mathbb{F}_q) = \frac{\phi(m)}{2}$.

\item In any other case, $C(m,\mathbb{F}_q) = 0$.

\end{enumerate}

\end{theorem}
\end{beisp}

%
%
%
%
%
%
%
%
%

\end{document}